\documentclass{amsart}
\usepackage{setspace}
\usepackage{a4}
\usepackage{amsthm}
\usepackage{latexsym}
\usepackage{amsfonts}
\usepackage{graphicx}
\usepackage{textcomp}
\usepackage{cite}
\usepackage{enumerate}
\usepackage{amssymb}
\usepackage{hyperref}
\usepackage{amsmath}
\usepackage{tikz}
\usepackage[mathscr]{euscript}
\usepackage{mathtools}
\newtheorem{theorem}{Theorem}[section]

\newtheorem{corollary}[theorem] {Corollary}
\newtheorem{definition}[theorem]{Definition}

\newtheorem{question}[theorem]{Question}
\title{This is the title}
\usepackage{fancyhdr}

\begin{document}
	\vspace{0.9cm}	
\hrule\hrule\hrule\hrule\hrule
\vspace{0.3cm}	
\begin{center}
{\bf{FUNCTIONAL KUPPINGER-DURISI-B\"{O}LCSKEI UNCERTAINTY PRINCIPLE}}\\
\vspace{0.3cm}
\hrule\hrule\hrule\hrule\hrule
\vspace{0.3cm}
\textbf{K. MAHESH KRISHNA}\\
Post Doctoral Fellow \\
Statistics and Mathematics Unit\\
Indian Statistical Institute, Bangalore Centre\\
Karnataka 560 059, India\\
Email: kmaheshak@gmail.com\\

Date: \today
\end{center}

\hrule\hrule
\vspace{0.5cm}
%--------------------------------------
\textbf{Abstract}: Let  $\mathcal{X}$ be a Banach space. Let $\{\tau_j\}_{j=1}^n, \{\omega_k\}_{k=1}^m\subseteq \mathcal{X}$ and $\{f_j\}_{j=1}^n$,  $\{g_k\}_{k=1}^m\subseteq \mathcal{X}^*$ satisfy $	|f_j(\tau_j)|\geq 1$ for all $ 1\leq j \leq n$, $|g_k(\omega_k)|\geq 1 $ for all $1\leq k \leq m$. If    $x \in \mathcal{X}\setminus \{0\}$ is such that $x=\theta_\tau\theta_f x=\theta_\omega\theta_g x$, then we show that 
\begin{align}\label{FKDB}
	\|\theta_fx\|_0\|\theta_gx\|_0\geq	\frac{\bigg[1-(\|\theta_fx\|_0-1)\max\limits_{1\leq j,r \leq n,j\neq r}|f_j(\tau_r)|\bigg]^+\bigg[1-(\|\theta_g x\|_0-1)\max\limits_{1\leq k,s \leq m,k\neq s}|g_k(\omega_s)|\bigg]^+}{\left(\displaystyle\max_{1\leq j \leq n, 1\leq k \leq m}|f_j(\omega_k)|\right)\left(\displaystyle\max_{1\leq j \leq n, 1\leq k \leq m}|g_k(\tau_j)|\right)}.
\end{align}	
We call Inequality (\ref{FKDB}) as \textbf{Functional Kuppinger-Durisi-B\"{o}lcskei Uncertainty Principle}. Inequality (\ref{FKDB}) improves the uncertainty principle obtained by Kuppinger, Durisi and B\"{o}lcskei \textit{[IEEE Trans. Inform. Theory  (2012)]} (which improved the Donoho-Stark-Elad-Bruckstein uncertainty principle \textit{[SIAM J. Appl. Math. (1989),  IEEE Trans. Inform. Theory (2002)]}). We also derive functional form of the uncertainity principle obtained by  Studer, Kuppinger, Pope and B\"{o}lcskei \textit{[EEE Trans. Inform. Theory (2012)]}.

\textbf{Keywords}:   Uncertainty Principle, Hilbert space,  Banach space.

\textbf{Mathematics Subject Classification (2020)}: 46A45, 46B45, 42C15.\\

\hrule

\tableofcontents
\hrule
\section{Introduction}
 Let $\hat{}: \mathbb{C}^d \to  \mathbb{C}^d$ be the Fourier transform. For $h \in \mathbb{C}^d$, let $\|h\|_0$ be the number of nonzero entries in $h$. It is correct to say that the progress of today's world is not possible without the following result of Donoho and Stark \cite{DONOHOSTARK}.
\begin{theorem} \cite{DONOHOSTARK}  (\textbf{Donoho-Stark Uncertainty Principle})  \label{DS}
	For every $d\in \mathbb{N}$, 
	\begin{align}\label{DSE}
		\left(\frac{\|h\|_0+\|\widehat{h}\|_0}{2}\right)^2	\geq \|h\|_0\|\widehat{h}\|_0	\geq d, \quad \forall h \in \mathbb{C}^d\setminus \{0\}.
	\end{align}
\end{theorem}
 Given a collection $\{\tau_j\}_{j=1}^n$ in a finite dimensional Hilbert space $\mathcal{H}$ over $\mathbb{K}$ ($\mathbb{R}$ or $\mathbb{C}$), we define 
 
\begin{align*}
	\theta_\tau: \mathcal{H} \ni h \mapsto \theta_\tau h \coloneqq (\langle h, \tau_j\rangle)_{j=1}^n \in \mathbb{K} ^n.
\end{align*}
Elad and Bruckstein generalized Inequality (\ref{DSE})  to pairs of orthonormal bases \cite{ELADBRUCKSTEIN}. 
\begin{theorem} \cite{ELADBRUCKSTEIN} (\textbf{Elad-Bruckstein Uncertainty Principle})  \label{EB}
	Let $\{\tau_j\}_{j=1}^n$,  $\{\omega_j\}_{j=1}^n$ be two orthonormal bases for a  finite dimensional Hilbert space $\mathcal{H}$. Then 
	\begin{align*}
		\left(\frac{\|\theta_\tau h\|_0+\|\theta_\omega h\|_0}{2}\right)^2	\geq \|\theta_\tau h\|_0\|\theta_\omega h\|_0\geq \frac{1}{\displaystyle\max_{1\leq j, k \leq n}|\langle\tau_j, \omega_k \rangle|^2}, \quad \forall h \in \mathcal{H}\setminus \{0\}.
	\end{align*}
\end{theorem}
For $a \in \mathbb{R}$, set  $a^+\coloneqq \max\{0,a\}$. Kuppinger, Durisi and B\"{o}lcskei showed that Theorem \ref{EB} can be improved to  unit norm vectors \cite{KUPPINGERDURISIBOLCSKEI}.
\begin{theorem}\cite{KUPPINGERDURISIBOLCSKEI} (\textbf{Kuppinger-Durisi-B\"{o}lcskei Uncertainty Principle})\label{KDB}
Let $\{\tau_j\}_{j=1}^n$,  $\{\omega_k\}_{k=1}^m$ be two collections of unit vectors  in a   finite dimensional Hilbert space $\mathcal{H}$. If  $h \in \mathcal{H}\setminus \{0\}$ is such that 
\begin{align}\label{VAL}
	h=\theta_\tau^*\theta_\tau h=\theta_\omega^*\theta_\omega h,
\end{align}
then 
\begin{align*}
\|\theta_\tau h\|_0\|\theta_\omega h\|_0\geq	\frac{\bigg[1-(\|\theta_\tau h\|_0-1)\max\limits_{1\leq j,r \leq n,j\neq r}|\langle\tau_j, \tau_r \rangle|\bigg]^+\bigg[1-(\|\theta_\omega h\|_0-1)\max\limits_{1\leq k,s \leq m,k\neq s}|\langle\omega_k, \omega_s \rangle|\bigg]^+}{\displaystyle\max_{1\leq j \leq n, 1\leq k \leq m}|\langle\tau_j, \omega_k \rangle|^2}.
\end{align*}
\end{theorem}
Let $0\leq \varepsilon <1$. Recall that \cite{DONOHOSTARK} a vector  $(a_j)_{j=1}^n \in \mathbb{K}^n$ is said to be \textbf{$\varepsilon$-concentrated on a subset $M\subseteq \{1, \dots, n\}$ w.r.t. 1-norm}  if 
\begin{align*}
\sum_{j \in M}|a_j|\geq (1-\varepsilon) \sum_{j=1}^{n}|a_j| \iff \varepsilon \sum_{j=1}^{n}|a_j|\geq \sum_{j \in M^c}|a_j|.
\end{align*}
 Theorem \ref{KDB} has been improved by  Studer, Kuppinger, Pope and B\"{o}lcskei \cite{STUDERKUPPINNGERPOPEBOLCSKEI}. In the following theorem and in rest of the paper, given a subset $M\subseteq \mathbb{N}$, the number of elements in $M$ is denoted by $o(M)$.
\begin{theorem}\cite{STUDERKUPPINNGERPOPEBOLCSKEI} (\textbf{Studer-Kuppinger-Pope-B\"{o}lcskei Uncertainty Principle})\label{SKPB}
Let $\{\tau_j\}_{j=1}^n$,  $\{\omega_k\}_{k=1}^m$ be two collections of unit vectors  in a   finite dimensional Hilbert space $\mathcal{H}$. Let  $h \in \mathcal{H}\setminus \{0\}$ be such that 
\begin{align*}
	h=\theta_\tau^*\theta_\tau h=\theta_\omega^*\theta_\omega h.
\end{align*}
If $\theta_\tau h$ is $\varepsilon$-concentrated on a subset $M\subseteq \{1, \dots, n\}$ w.r.t. 1-norm  and  $\theta_\omega h$ is $\delta$-concentrated on a subset $N\subseteq \{1, \dots, n\}$ w.r.t. 1-norm, then 
\begin{align*}
	o(M)o(N)\geq	\frac{\bigg[1-\varepsilon-(o(M)-1+\varepsilon)\max\limits_{1\leq j,r \leq n,j\neq r}|\langle\tau_j, \tau_r \rangle|\bigg]^+\bigg[1-\delta-(o(N)-1+\delta)\max\limits_{1\leq k,s \leq m,k\neq s}|\langle\omega_k, \omega_s \rangle|\bigg]^+}{\displaystyle\max_{1\leq j \leq n, 1\leq k \leq m}|\langle\tau_j, \omega_k \rangle|^2}.
\end{align*}	
\end{theorem}
When $\varepsilon=0$, Theorem \ref{SKPB} reduces to Theorem \ref{KDB}.
In this paper, we derive both finite and infinite dimensional Banach space versions of Theorems \ref{KDB} and \ref{SKPB}. It is reasonable to note that Theorem \ref{EB} has been improved using Parseval frames for Hilbert spaces by 
Ricaud and Torr\'{e}sani \cite{RICAUDTORRESANI} and later extended to Banach spaces in the paper \cite{KRISHNA}. Most important thing to keep in mind is that uncertainty principle derived in \cite{RICAUDTORRESANI} is for Parseval frames (which says vectors have norm less than or equal to one) which is not required in Theorem \ref{KDB} (but with the condition that vectors are unit vectors). Also note that it is not required the validity of Equation (\ref{VAL}) for all $h\in \mathcal{H}$ (in that case, both will become orthonormal bases).

\section{Functional Kuppinger-Durisi-B\"{o}lcskei   Uncertainty Principle}
In the paper,   $\mathbb{K}$ denotes $\mathbb{C}$ or $\mathbb{R}$ and $\mathcal{X}$ denotes a   Banach space  over $\mathbb{K}$. Dual of $\mathcal{X}$ is denoted by $\mathcal{X}^*$. Given a collection $\{\tau_j\}_{j=1}^n$ in  $\mathcal{X}$ and a collection $\{f_j\}_{j=1}^n$ in  $\mathcal{X}^*$   we define 
\begin{align*}
	&\theta_f: \mathcal{X} \ni x \mapsto \theta_fx \coloneqq (f_j(x))_{j=1}^n \in \mathbb{K} ^n,\\
	&\theta_\tau:\mathbb{K} ^n \ni (a_j)_{j=1}^n \mapsto \sum_{j=1}^{n}a_j\tau_j \in \mathcal{X}.
\end{align*}
Following is the Banach space generalization of Theorem \ref{KDB}.
\begin{theorem}
(\textbf{Functional Kuppinger-Durisi-B\"{o}lcskei Uncertainty Principle}) \label{FKDBB}
Let $\{\tau_j\}_{j=1}^n$,  $\{\omega_k\}_{k=1}^m$ be two  collections in a   finite dimensional Banach  space $\mathcal{X}$ and $\{f_j\}_{j=1}^n$,  $\{g_k\}_{k=1}^m$ be two  collections in  $\mathcal{X}^*$ satisfying
\begin{align*}
	|f_j(\tau_j)|\geq 1,~ \forall 1\leq j \leq n, \quad |g_k(\omega_k)|\geq 1,~ \forall 1\leq k \leq m.
\end{align*}
 If  $x \in \mathcal{X}\setminus \{0\}$ is such that 
\begin{align}\label{ASSU}
	x=\theta_\tau\theta_f x=\theta_\omega\theta_g x,
\end{align}
then 
\begin{align*}
	\|\theta_fx\|_0\|\theta_gx\|_0\geq	\frac{\bigg[1-(\|\theta_fx\|_0-1)\max\limits_{1\leq j,r \leq n,j\neq r}|f_j(\tau_r)|\bigg]^+\bigg[1-(\|\theta_g x\|_0-1)\max\limits_{1\leq k,s \leq m,k\neq s}|g_k(\omega_s)|\bigg]^+}{\left(\displaystyle\max_{1\leq j \leq n, 1\leq k \leq m}|f_j(\omega_k)|\right)\left(\displaystyle\max_{1\leq j \leq n, 1\leq k \leq m}|g_k(\tau_j)|\right)}.
\end{align*}	
\end{theorem}
\begin{proof}
	Let $1\leq j \leq n$. Then using Equation (\ref{ASSU})
	\begin{align*}
		|f_j(x)|&=	|f_j(\theta_\tau\theta_f x)|=\left|f_j\left(\sum_{r=1}^{n}f_r(x)\tau_r\right)\right|=\left|\sum_{r=1}^{n}f_r(x)f_j(\tau_r)\right|\\
		&=\left|f_j(x)f_j(\tau_j)+\sum_{r=1, r\neq j}^{n}f_r(x)f_j(\tau_r)\right|\geq \left|f_j(x)f_j(\tau_j)\right|-\left|\sum_{r=1, r\neq j}^{n}f_r(x)f_j(\tau_r)\right|\\
		&\geq \left|f_j(x)\right|-\left|\sum_{r=1, r\neq j}^{n}f_r(x)f_j(\tau_r)\right|\geq \left|f_j(x)\right|-\sum_{r=1, r\neq j}^{n}|f_r(x)f_j(\tau_r)|\\
		&\geq \left|f_j(x)\right|-\left(\sum_{r=1, r\neq j}^{n}|f_r(x)|\right)\max\limits_{1\leq j,r \leq n,j\neq r}|f_j(\tau_r)|\\
		&=\left|f_j(x)\right|-\left(\sum_{r=1}^{n}|f_r(x)|-|f_j(x)|\right)\max\limits_{1\leq j,r \leq n,j\neq r}|f_j(\tau_r)|\\
		&=\left|f_j(x)\right|-\left(\|\theta_fx\|_1-|f_j(x)|\right)\max\limits_{1\leq j,r \leq n,j\neq r}|f_j(\tau_r)|\\
		&=\left(1+\max\limits_{1\leq j,r \leq n,j\neq r}|f_j(\tau_r)|\right)\left|f_j(x)\right|-\|\theta_fx\|_1\max\limits_{1\leq j,r \leq n,j\neq r}|f_j(\tau_r)|.
	\end{align*}
On the other hand, again using Equation (\ref{ASSU})

\begin{align*}
	|f_j(x)|&=	|f_j(\theta_\omega\theta_g x)|=\left|f_j\left(\sum_{k=1}^{m}g_k(x)\omega_k\right)\right|=\left|\sum_{k=1}^{m}g_k(x)f_j(\omega_k)\right|\\
	&\leq \sum_{k=1}^{m}|g_k(x)f_j(\omega_k)|\leq \left(\sum_{k=1}^{m}|g_k(x)|\right)\displaystyle\max_{1\leq j \leq n, 1\leq k \leq m}|f_j(\omega_k)|\\
	&=\|\theta_gx\|_1\displaystyle\max_{1\leq j \leq n, 1\leq k \leq m}|f_j(\omega_k)|
\end{align*}
Therefore we have 
\begin{align}\label{I1}
\left(1+\max\limits_{1\leq j,r \leq n,j\neq r}|f_j(\tau_r)|\right)\left|f_j(x)\right|-\|\theta_fx\|_1\max\limits_{1\leq j,r \leq n,j\neq r}|f_j(\tau_r)|\leq \|\theta_g x \|_1\displaystyle\max_{1\leq j \leq n, 1\leq k \leq m}|f_j(\omega_k)|
\end{align}
Summing Inequality (\ref{I1})  on  the support of $\theta_f x$ we get 
\begin{align*}
&\left(1+\max\limits_{1\leq j,r \leq n,j\neq r}|f_j(\tau_r)|\right)\sum_{j \in \operatorname{supp}(\theta_fx)}\left|f_j(x)\right|-\|\theta_fx\|_1\left(\max\limits_{1\leq j,r \leq n,j\neq r}|f_j(\tau_r)|\right)\sum_{j \in \operatorname{supp}(\theta_fx)}1\leq \\
&\quad \|\theta_g x \|_1\left(\displaystyle\max_{1\leq j \leq n, 1\leq k \leq m}|f_j(\omega_k)|\right)\sum_{j \in \operatorname{supp}(\theta_fx)}1,
\end{align*}
i.e., 
\begin{align*}
	&\left(1+\max\limits_{1\leq j,r \leq n,j\neq r}|f_j(\tau_r)|\right)\|\theta_fx\|_1-\|\theta_fx\|_1\left(\max\limits_{1\leq j,r \leq n,j\neq r}|f_j(\tau_r)|\right)\|\theta_fx\|_0\leq \\
	&\quad \|\theta_g x \|_1\left(\displaystyle\max_{1\leq j \leq n, 1\leq k \leq m}|f_j(\omega_k)|\right)\|\theta_fx\|_0,
\end{align*}
i.e., 
\begin{align*}
\left[1-(\|\theta_f x\|_0-1)\max\limits_{1\leq j,r \leq n,j\neq r}|f_j(\tau_r)|\right]\|\theta_fx\|_1\leq 	\|\theta_g x \|_1\left(\displaystyle\max_{1\leq j \leq n, 1\leq k \leq m}|f_j(\omega_k)|\right)\|\theta_fx\|_0.
\end{align*}
Since the right side of previous inequality is non negative, we have
\begin{align}\label{PI}
	\left[1-(\|\theta_f x\|_0-1)\max\limits_{1\leq j,r \leq n,j\neq r}|f_j(\tau_r)|\right]^+\|\theta_fx\|_1\leq 	\|\theta_g x \|_1\left(\displaystyle\max_{1\leq j \leq n, 1\leq k \leq m}|f_j(\omega_k)|\right)\|\theta_fx\|_0.
\end{align}
Similarly
\begin{align}\label{P2}
\bigg[1-(\|\theta_g x\|_0-1)\max\limits_{1\leq k,s \leq m,k\neq s}|g_k(\omega_s)|\bigg]^+\|\theta_gx\|_1\leq \|\theta_f x \|_1	\left(\displaystyle\max_{1\leq j \leq n, 1\leq k \leq m}|g_k(\tau_j)|\right)\|\theta_gx\|_0.
\end{align}
Multiplying Inequalities (\ref{PI}) and (\ref{P2}) we get
\begin{align*}
&\left[1-(\|\theta_f x\|_0-1)\max\limits_{1\leq j,r \leq n,j\neq r}|f_j(\tau_r)|\right]^+\bigg[1-(\|\theta_g x\|_0-1)\max\limits_{1\leq k,s \leq m,k\neq s}|g_k(\omega_s)|\bigg]^+	\|\theta_fx\|_1 \|\theta_gx\|_1\leq \\
&\quad \|\theta_gx\|_1 \|\theta_fx\|_1\|\theta_fx\|_0 \|\theta_gx\|_0\left(\displaystyle\max_{1\leq j \leq n, 1\leq k \leq m}|f_j(\omega_k)|\right)\left(\displaystyle\max_{1\leq j \leq n, 1\leq k \leq m}|g_k(\tau_j)|\right).
\end{align*}
A cancellation of $\|\theta_f x \|_1\|\theta_g x \|_1$  gives the required inequality.
\end{proof}
Next we derive Banach space version of Theorem \ref{SKPB}.
\begin{theorem}\label{FSKPB}
(\textbf{Functional Studer-Kuppinger-Pope-B\"{o}lcskei  Uncertainty Principle})
Let $\{\tau_j\}_{j=1}^n$,  $\{\omega_k\}_{k=1}^m$ be two  collections in a   finite dimensional Banach  space $\mathcal{X}$ and $\{f_j\}_{j=1}^n$,  $\{g_k\}_{k=1}^m$ be two  collections in  $\mathcal{X}^*$ satisfying

\begin{align}\label{R}
	|f_j(\tau_j)|\geq 1,~ \forall 1\leq j \leq n, \quad |g_k(\omega_k)|\geq 1,~ \forall 1\leq k \leq m.
\end{align}
Let   $x \in \mathcal{X}\setminus \{0\}$ be  such that 
\begin{align}\label{R2}
	x=\theta_\tau\theta_f x=\theta_\omega\theta_g x.
\end{align}
If $\theta_fx$ is $\varepsilon$-concentrated on a subset $M\subseteq \{1, \dots, n\}$ w.r.t. 1-norm  and  $\theta_g x$ is $\delta$-concentrated on a subset $N\subseteq \{1, \dots, n\}$ w.r.t. 1-norm, then 
\begin{align}\label{R3}
o(M)o(N)\geq	\frac{\bigg[1-\varepsilon-(o(M)-1+\varepsilon)\max\limits_{1\leq j,r \leq n,j\neq r}|f_j(\tau_r)|\bigg]^+\bigg[1-\delta-(o(N)-1+\delta)\max\limits_{1\leq k,s \leq m,k\neq s}|g_k(\omega_s)|\bigg]^+}{\left(\displaystyle\max_{1\leq j \leq n, 1\leq k \leq m}|f_j(\omega_k)|\right)\left(\displaystyle\max_{1\leq j \leq n, 1\leq k \leq m}|g_k(\tau_j)|\right)}.
\end{align}		
\end{theorem}
\begin{proof}
We start by using Equation (\ref{I1}). Let $1\leq j \leq n$. Then 
\begin{align}\label{Il1}
	\left(1+\max\limits_{1\leq j,r \leq n,j\neq r}|f_j(\tau_r)|\right)\left|f_j(x)\right|-\|\theta_fx\|_1\max\limits_{1\leq j,r \leq n,j\neq r}|f_j(\tau_r)|\leq \|\theta_g x \|_1\displaystyle\max_{1\leq j \leq n, 1\leq k \leq m}|f_j(\omega_k)|
\end{align}
Summing Inequality (\ref{Il1})  on  the support of $M$ we get 
\begin{align*}
	&\left(1+\max\limits_{1\leq j,r \leq n,j\neq r}|f_j(\tau_r)|\right)\sum_{j \in M}\left|f_j(x)\right|-\|\theta_fx\|_1\left(\max\limits_{1\leq j,r \leq n,j\neq r}|f_j(\tau_r)|\right)\sum_{j \in M}1\leq \\
	&\quad \|\theta_g x \|_1\left(\displaystyle\max_{1\leq j \leq n, 1\leq k \leq m}|f_j(\omega_k)|\right)\sum_{j \in M}1,
\end{align*}
i.e., 
\begin{align}\label{12}
	&\left(1+\max\limits_{1\leq j,r \leq n,j\neq r}|f_j(\tau_r)|\right)\sum_{j \in M}\left|f_j(x)\right|-\|\theta_fx\|_1\left(\max\limits_{1\leq j,r \leq n,j\neq r}|f_j(\tau_r)|\right)o(M)\leq \nonumber\\
	&\quad \|\theta_g x \|_1\left(\displaystyle\max_{1\leq j \leq n, 1\leq k \leq m}|f_j(\omega_k)|\right)o(M).
\end{align}
Since $\theta_fx$ is $\varepsilon$-concentrated on $M$ we are given with 
\begin{align}\label{13}
	\sum_{j \in M}|f_j(x)|\geq (1-\varepsilon) \sum_{j=1}^{n}|f_j(x)|.
\end{align}
Using Inequality (\ref{13}) in Inequality (\ref{12}) we get
\begin{align*}
&\left(1+\max\limits_{1\leq j,r \leq n,j\neq r}|f_j(\tau_r)|\right)(1-\varepsilon) \sum_{j=1}^{n}|f_j(x)|-\|\theta_fx\|_1\left(\max\limits_{1\leq j,r \leq n,j\neq r}|f_j(\tau_r)|\right)o(M)	\leq \\
&\left(1+\max\limits_{1\leq j,r \leq n,j\neq r}|f_j(\tau_r)|\right)\sum_{j \in M}\left|f_j(x)\right|-\|\theta_fx\|_1\left(\max\limits_{1\leq j,r \leq n,j\neq r}|f_j(\tau_r)|\right)o(M)\leq  \\
&\|\theta_g x \|_1\left(\displaystyle\max_{1\leq j \leq n, 1\leq k \leq m}|f_j(\omega_k)|\right)o(M),	
\end{align*}
i.e., 

\begin{align*}
&\left(1+\max\limits_{1\leq j,r \leq n,j\neq r}|f_j(\tau_r)|\right)(1-\varepsilon) \|\theta_fx\|_1-\|\theta_fx\|_1\left(\max\limits_{1\leq j,r \leq n,j\neq r}|f_j(\tau_r)|\right)o(M)	\leq\\ &\|\theta_g x \|_1\left(\displaystyle\max_{1\leq j \leq n, 1\leq k \leq m}|f_j(\omega_k)|\right)o(M),
\end{align*}
i.e., 
\begin{align*}
	\left[1-\varepsilon-(o(M)-1+\varepsilon)\max\limits_{1\leq j,r \leq n,j\neq r}|f_j(\tau_r)|\right]\|\theta_fx\|_1\leq 	\|\theta_g x \|_1\left(\displaystyle\max_{1\leq j \leq n, 1\leq k \leq m}|f_j(\omega_k)|\right)o(M).
\end{align*}
Since the right side of previous inequality is non negative, we have
\begin{align}\label{PI1}
	\left[1-\varepsilon-(o(M)-1+\varepsilon)\max\limits_{1\leq j,r \leq n,j\neq r}|f_j(\tau_r)|\right]^+\|\theta_fx\|_1\leq 	\|\theta_g x \|_1\left(\displaystyle\max_{1\leq j \leq n, 1\leq k \leq m}|f_j(\omega_k)|\right)o(M).
\end{align}
Similarly
\begin{align}\label{P22}
	\bigg[1-\delta-(o(N)-1+\delta)\max\limits_{1\leq k,s \leq m,k\neq s}|g_k(\omega_s)|\bigg]^+\|\theta_gx\|_1\leq \|\theta_f x \|_1	\left(\displaystyle\max_{1\leq j \leq n, 1\leq k \leq m}|g_k(\tau_j)|\right)o(N).
\end{align}
Multiplying Inequalities (\ref{PI1}) and (\ref{P22}) we get
\begin{align*}
	&\left[1-\varepsilon-(o(M)-1+\varepsilon)\max\limits_{1\leq j,r \leq n,j\neq r}|f_j(\tau_r)|\right]^+\bigg[1-\delta-(o(N)-1+\delta)\max\limits_{1\leq k,s \leq m,k\neq s}|g_k(\omega_s)|\bigg]^+	\|\theta_fx\|_1 \|\theta_gx\|_1\leq \\
	&\quad \|\theta_gx\|_1 \|\theta_fx\|_1\|o(M)o(N)\left(\displaystyle\max_{1\leq j \leq n, 1\leq k \leq m}|f_j(\omega_k)|\right)\left(\displaystyle\max_{1\leq j \leq n, 1\leq k \leq m}|g_k(\tau_j)|\right).
\end{align*}
By canceling $\|\theta_f x \|_1\|\theta_g x \|_1$ we get the  inequality in the statement of theorem.
\end{proof}
Note that $\theta_fx $ (resp. $\theta_gx $) is 0-supported on $\operatorname{supp}(\theta_fx)$ (resp. $\operatorname{supp}(\theta_gx)$). Hence Theorem \ref{FKDBB} follows from Theorem \ref{FSKPB}.
\begin{corollary}
	Theorem  \ref{SKPB}  follows from Theorem \ref{FSKPB}.
\end{corollary}
\begin{proof}
Given two collections $\{\tau_j\}_{j=1}^n$,  $\{\omega_k\}_{k=1}^m$  of unit vectors  in a   finite dimensional Hilbert space $\mathcal{H}$, by defining 
	\begin{align*}
		f_j:\mathcal{H} \ni h \mapsto \langle h, \tau_j \rangle \in \mathbb{K}; \quad \forall 1\leq j\leq n, \quad g_k:\mathcal{H} \ni h \mapsto \langle h, \omega_k \rangle \in \mathbb{K}, \quad \forall 1\leq k\leq m
	\end{align*}
	we get the result.
\end{proof}
Theorem  \ref{FSKPB}  brings the following question.
\begin{question}
	Given a Banach space $\mathcal{X}$  for which subsets $M,N\subseteq \mathbb{N}$ and pairs  $(\{f_j\}_{j=1}^n, \{\tau_j\}_{j=1}^n)$,  $(\{g_k\}_{k=1}^m, \{\omega_k\}_{k=1}^m)$ satisfying (\ref{R})  and (\ref{R2}) we have equality in Inequality (\ref{R3})?
\end{question}

\section{Infinite dimensional Functional Kuppinger-Durisi-B\"{o}lcskei   Uncertainty Principle}
In this section we derive infinite dimensional versions of Theorem \ref{FKDBB} and Theorem \ref{FSKPB}.  Unlike finite dimensions, we cannot start with arbitrary infinite collection of elements in a Banach space. Following restricted class of collection has to be used. 
\begin{definition}\cite{KRISHNAJOHNSON}
	Let $\mathcal{X}$ be a  Banach space, $\{\tau_j\}_{j=1}^\infty \subseteq \mathcal{X}$ and $\{f_j\}_{j=1}^\infty \subseteq \mathcal{X}^*$. The   pair $(\{f_j\}_{j=1}^\infty, \{\tau_j\}_{j=1}^\infty)$ is said to be a \textbf{1-approximate Bessel sequence} (1-ABS) for $\mathcal{X}$ if following conditions hold.
	\begin{enumerate}[\upshape(i)]
		\item The map 
		
		\begin{align*}
			&\theta_f: \mathcal{X} \ni x \mapsto \theta_fx \coloneqq \{f_j(x)\}_{j=1}^\infty \in \ell^1(\mathbb{N})
		\end{align*}
	is a well-defined bounded linear operator. 
		\item The map 
		\begin{align*}
			&\theta_\tau:\ell^1(\mathbb{N}) \ni \{a_j\}_{j=1}^\infty \mapsto \sum_{j=1}^{\infty}a_j\tau_j \in \mathcal{X}
		\end{align*}
		is a well-defined bounded linear operator. 
	\end{enumerate}
\end{definition}
\begin{theorem}
Let $(\{f_j\}_{j=1}^\infty, \{\tau_j\}_{j=1}^\infty)$ and   $(\{g_k\}_{k=1}^\infty, \{\omega_k\}_{k=1}^\infty )$ be two  1-ABS for a  Banach  space $\mathcal{X}$  satisfying
\begin{align*}
	|f_j(\tau_j)|\geq 1,~ \forall j\in \mathbb{N}, \quad |g_k(\omega_k)|\geq 1,~ \forall k \in \mathbb{N}.
\end{align*}
If  $x \in \mathcal{X}\setminus \{0\}$ is such that 
\begin{align*}
	x=\theta_\tau\theta_f x=\theta_\omega\theta_g x,
\end{align*}
then 
\begin{align*}
	\|\theta_fx\|_0\|\theta_gx\|_0\geq	\frac{\bigg[1-(\|\theta_fx\|_0-1)\sup\limits_{j,r \in \mathbb{N},j\neq r}|f_j(\tau_r)|\bigg]^+\bigg[1-(\|\theta_g x\|_0-1)\sup\limits_{k,s \in \mathbb{N},k\neq s}|g_k(\omega_s)|\bigg]^+}{\left(\displaystyle\sup_{j,  k \in \mathbb{N} }|f_j(\omega_k)|\right)\left(\displaystyle\sup_{j, k \in \mathbb{N}}|g_k(\tau_j)|\right)}.
\end{align*}		
\end{theorem}
\begin{proof}
	Let $j \in \mathbb{N}$. Then 
\begin{align*}
	|f_j(x)|&=	|f_j(\theta_\tau\theta_f x)|=\left|f_j\left(\sum_{r=1}^{\infty}f_r(x)\tau_r\right)\right|=\left|\sum_{r=1}^{\infty}f_r(x)f_j(\tau_r)\right|\\
	&=\left|f_j(x)f_j(\tau_j)+\sum_{r=1, r\neq j}^{\infty}f_r(x)f_j(\tau_r)\right|\geq \left|f_j(x)f_j(\tau_j)\right|-\left|\sum_{r=1, r\neq j}^{\infty}f_r(x)f_j(\tau_r)\right|\\
	&\geq \left|f_j(x)\right|-\left|\sum_{r=1, r\neq j}^{\infty}f_r(x)f_j(\tau_r)\right|\geq \left|f_j(x)\right|-\sum_{r=1, r\neq j}^{\infty}|f_r(x)f_j(\tau_r)|\\
	&\geq \left|f_j(x)\right|-\left(\sum_{r=1, r\neq j}^{\infty}|f_r(x)|\right)\sup\limits_{ j,r \in  \mathbb{N},j\neq r}|f_j(\tau_r)|\\
	&=\left|f_j(x)\right|-\left(\sum_{r=1}^{\infty}|f_r(x)|-|f_j(x)|\right)\sup\limits_{j,r \in \mathbb{N},j\neq r}|f_j(\tau_r)|\\
	&=\left|f_j(x)\right|-\left(\|\theta_fx\|_1-|f_j(x)|\right)\sup\limits_{j,r \in \mathbb{N},j\neq r}|f_j(\tau_r)|\\
	&=\left(1+\sup\limits_{1\leq j,r \leq n,j\neq r}|f_j(\tau_r)|\right)\left|f_j(x)\right|-\|\theta_fx\|_1\sup\limits_{j,r \in \mathbb{N},j\neq r}|f_j(\tau_r)|.
\end{align*}
We also find 

\begin{align*}
	|f_j(x)|&=	|f_j(\theta_\omega\theta_g x)|=\left|f_j\left(\sum_{k=1}^{\infty}g_k(x)\omega_k\right)\right|=\left|\sum_{k=1}^{\infty}g_k(x)f_j(\omega_k)\right|\\
	&\leq \sum_{k=1}^{\infty}|g_k(x)f_j(\omega_k)|\leq \left(\sum_{k=1}^{\infty}|g_k(x)|\right)\displaystyle\sup_{j, k\in \mathbb{N}}|f_j(\omega_k)|\\
	&=\|\theta_gx\|_1\displaystyle\sup_{j, k \in \mathbb{N}}|f_j(\omega_k)|.
\end{align*}
	Now by doing a similar type of calculation as in the proof of Theorem \ref{FKDBB} we get the result.
\end{proof}
We recall that a vector  $\{a_j\}_{j=1}^\infty \in \ell^1(\mathbb{N})$ is said to be $\varepsilon$-concentrated on a subset $M\subseteq \mathbb{N}$ w.r.t. 1-norm  if 
\begin{align*}
	\sum_{j \in M}|a_j|\geq (1-\varepsilon) \sum_{j=1}^{\infty}|a_j| \iff \varepsilon \sum_{j=1}^{\infty}|a_j|\geq \sum_{j \in M^c}|a_j|.
\end{align*}
It is a easy to see the following infinite dimensional version of Theorem \ref{FSKPB}.
\begin{theorem}
Let $(\{f_j\}_{j=1}^\infty, \{\tau_j\}_{j=1}^\infty)$ and   $(\{g_k\}_{k=1}^\infty, \{\omega_k\}_{k=1}^\infty )$ be two  1-ABS for a  Banach  space $\mathcal{X}$  satisfying
\begin{align*}
	|f_j(\tau_j)|\geq 1,~ \forall j\in \mathbb{N}, \quad |g_k(\omega_k)|\geq 1,~ \forall k \in \mathbb{N}.
\end{align*}
	Let   $x \in \mathcal{X}\setminus \{0\}$ be  such that 
	\begin{align*}
		x=\theta_\tau\theta_f x=\theta_\omega\theta_g x.
	\end{align*}
	If $\theta_fx$ is $\varepsilon$-concentrated on a subset $M\subseteq \mathbb{N}$ w.r.t. 1-norm  and  $\theta_g x$ is $\delta$-concentrated on a subset $N\subseteq \mathbb{N}$ w.r.t. 1-norm, then 
	\begin{align*}
		o(M)o(N)\geq	\frac{\bigg[1-\varepsilon-(o(M)-1+\varepsilon)\sup\limits_{j,r \in  \mathbb{N},j\neq r}|f_j(\tau_r)|\bigg]^+\bigg[1-\delta-(o(N)-1+\delta)\sup\limits_{k,s \in  \mathbb{N},k\neq s}|g_k(\omega_s)|\bigg]^+}{\left(\displaystyle\sup_{j, k \in  \mathbb{N} }|f_j(\omega_k)|\right)\left(\displaystyle\sup_{j, k \in \mathbb{N}}|g_k(\tau_j)|\right)}.
	\end{align*}		
\end{theorem}
The techniques used in \cite{KRISHNA} have  been extended to derive continuous versions of uncertainty principles for Banach spaces using Lebesgue function spaces \cite{KRISHNA2}. However, it seems that the techniques used in this paper cannot be extended to get continuous versions of the results derived in this paper.

We end the paper with the following two interesting and important questions.
\begin{question}
	\begin{enumerate}[\upshape(i)]
		\item Can Theorem \ref{FSKPB} be improved using divisors of the dimension of the space (like Roy uncertainty principle \cite{ROY1, ROY2}, Murty-Whang  uncertainty principle \cite{MURTYWHANG}). In particular,  for  prime dimensional Banach spaces (like Tao uncertainty principle \cite{TAO})?
		\item What are the versions  Theorem \ref{FSKPB} and the results in  \cite{KRISHNA}  for vector spaces over finite fields (like Goldstein-Guralnick-Isaacs uncertainty principle \cite{GOLDSTEINGURALNICKISAACS}, Evra-Kowalski-Lubotzky uncertainty principle \cite{EVRAKOWALSKILUBOTZKY}, Borello-Willems-Zini  uncertainty principle \cite{BORELLOWILLEMSZINI}, Feng-Hollmann-Xiang uncertainty principle \cite{FENGHOLLMANNXIANG}, Garcia-Karaali-Katz uncertainty principle \cite{GARCIAKARAALIKATZ} and  Borello-Sol\'{e} uncertainty principle \cite{BORELLOSOLE})?
	\end{enumerate}
\end{question}

 \bibliographystyle{plain}
 \bibliography{reference.bib}

\end{document}